\documentclass[11pt,a4paper]{article}
\usepackage{cite}
\usepackage{amsmath}
\usepackage{cases}

\usepackage{epsf,epsfig,amsfonts,amsgen,amsmath,amstext,amsbsy,amsopn,amsthm
}
\usepackage{enumerate}
\usepackage{amsmath}
\usepackage{amsfonts,amsthm,amssymb}
\usepackage{amsfonts}
\usepackage{graphics}
\usepackage{latexsym,bm}
\usepackage{amsfonts,amsthm,amssymb,bbding}
\usepackage{indentfirst}
\usepackage{graphicx}
\usepackage{color}
\usepackage{moresize}
\usepackage{array}
\usepackage{multirow}
\usepackage{graphicx}
\usepackage{threeparttable}
\usepackage{makecell}
\renewcommand\proofname{\bf Proof}
\newtheorem{theorem}{Theorem}

\newtheorem{lemma}{Lemma}[section]

\newtheorem{false statement}{False statement}

\theoremstyle{definition}

\newtheorem{claim}{Claim}

\newtheorem{remark}[claim]{Remark}

\newtheorem{problem}{Problem}

\usepackage{amsmath}
\allowdisplaybreaks[4]
\begin{document}

\title{\bf\Large Spectral radius, toughness and $k$-factor of graphs}
\author{\Large Yuanyuan Chen$^{a,b}$, Huiqiu Lin$^{a,b}$\thanks{Corresponding author;
Email addresses:
chenyy\underline{\ }de@sina.com (Y. Chen), huiqiulin@126.com (H. Lin), shuchengli666@163.com (S. Li).}, Shucheng Li$^{b}$ \\
\small\it $^{a}$  School of Mathematics, East China University of Science and Technology, \\
\small\it   Shanghai 200237, China\\[1mm]
\small\it $^b$ College of Mathematics and System Science, Xinjiang University,\\
\small\it  Urumqi, Xinjiang 830017, China\\[1mm]}
\date{}

\maketitle

\begin{center}
\begin{minipage}{140mm}
\begin{center}{\bf Abstract}\end{center}
{A $k$-regular spanning subgraph of $G$ is called a $k$-factor.
Fan, Lin and Lu [European J. Combin. 110 (2023) 103701] presented a tight sufficient condition
in terms of the spectral radius for a connected 1-tough graph to contain a connected 2-factor (Hamilton cycle).
Then it is interesting to consider the following problem: What is the spectral radius condition to guarantee the existence of
a $k$-factor with $k\ge3$ in a connected 1-tough graph $G$ with $\delta(G)\ge k$?
In this paper, we completely solve this problem.

}
\end{minipage}

\end{center}

\renewcommand{\thefootnote}{}
\vskip 0.3in

\section{Introduction}
Throughout this paper, we consider only finite, undirected and simple connected graphs. For a vertex $v\in V(G)$, let $N_G(v)$ and  $d_{G} (v)$ be the neighborhood and degree of $v$ in $G$, respectively.
The largest eigenvalue of $A(G)$, denoted by $\rho(G)$, is called the \emph{spectral radius} of $G$.
Given two graphs $G_1$ and $G_2$, the \textit{disjoint union} $G_1\cup G_2$ is the graph with vertex set $V(G_1)\cup V(G_2)$ and edge set $E(G_1)\cup E(G_2)$,
and the \textit{join}  $G_{1} \vee G_{2}$ is the graph obtained from  $G_{1}\cup G_{2}$ by adding all edges between $G_{1}$ and $G_{2}$.

An $[a, b]$-factor of a graph $G$ is a spanning subgraph $H$ such that $a\le d_H(v)\le b$ for each $v\in V(G)$.
In particular, for a positive integer $k$, a $[k,k]$-factor is called a $k$-factor.
The initial study of factors was due to Danish mathematician Petersen \cite{J.Petersen} in 1891. After that, many researches have been conducted on this topic.
There are abundant achievements of studying the existence of factors in graphs from the spectral perspective. Brouwer
and Haemers \cite{Brouwer1} initiated the research of establishing sufficient conditions for a regular graph to have a perfect matching in
terms of the third largest adjacency eigenvalue. Their result was improved in \cite{S.Cioab1, S.Cioab2, S.Cioab3} and extended in \cite{H.Lu1, H.Lu2} to obtain a regular factor.
Recently, the researchers have also paid much attention to the existence of factors in a graph from the prospective of the
spectral radius, which can be seen in
\cite{Brouwer1, Barik, S.Cioab1, S.Cioab2, S.Cioab3, E.Cho, Y.Cui, M.Ellingham, D.Fan3, Gu2, Kano1, Kim, H.Lu1, H.Lu2, H.Lu3, J.Wei} and the references therein, among others.

Chv\'atal \cite{V.C} defined the \textit{toughness} of a non-complete graph $G$ as
$t(G)=\min\{\frac{|S|}{c(G-S)}: S\subset V(G) ~\textrm{and}~ c(G-S)>1\},$ where $c(G-S)$ denotes the number of components of $G-S$.
A graph $G$ is $t$-tough if $|S|\ge tc(G-S)$ for every $S\subseteq V(G)$ with $c(G-S)>1$.
Many researchers also focused on the existence of factors in a graph from the prospective of toughness, see \cite{Alon, D.Bauer, Y.Cui, CFL241, CGL26, D.Fan2, R.Liu}.
A Hamiltonian cycle is a connected 2-factor of a graph.
Let $K^{+3}_{n-4}$ be the graph obtained from $3K_1\cup K_{n-4}$ by adding three independent edges between $3K_1$ and $K_{n-4}$, and let $M_n=K_1\vee K^{+3}_{n-4}.$
Fan, Lin and Lu \cite{D.Fan2} determined the spectral radius condition to guarantee the existence of a Hamiltonian cycle (connected 2-factor) among 1-tough graphs.
\begin{theorem}[{Fan and Lin \cite{D.Fan2}}]\label{tF2}
Suppose that $G$ is a connected 1-tough graph of order $n\ge 18$ with $\delta(G)\ge2$.
If $\rho(G)\ge\rho(M_n)$, then $G$ contains a Hamiltonian cycle, unless $G\cong M_n$.
\end{theorem}

Then it is interesting to consider the following problem:
\begin{problem}\label{k-tough}
What is the spectral radius condition to guarantee the existence of $k(\ge3)$-factor in connected 1-tough graph $G$ with $\delta(G)\ge k$?
\end{problem}
Suppose that $V((k+1)K_1)=\{v_1,\ldots, v_{k+1}\}$, $V(K_{n-2k-1})=\{w_1,\ldots, w_{n-2k-1}\}$.
Let $G^1_{n,3}=K_3\vee (4K_1\cup K_{n-7})+\{v_1w_1, v_2w_2\}$ as shown in Fig.\,1 (a),
and let $G_{n,k}=K_k\vee ((k+1)K_1\cup K_{n-2k-1})+\{v_1w_i|i=1,\ldots,k-2\}+\{v_2w_1\}$ as shown in Fig.\,1 (b).
\begin{figure}
    \centering
    \includegraphics[width=0.85\linewidth]{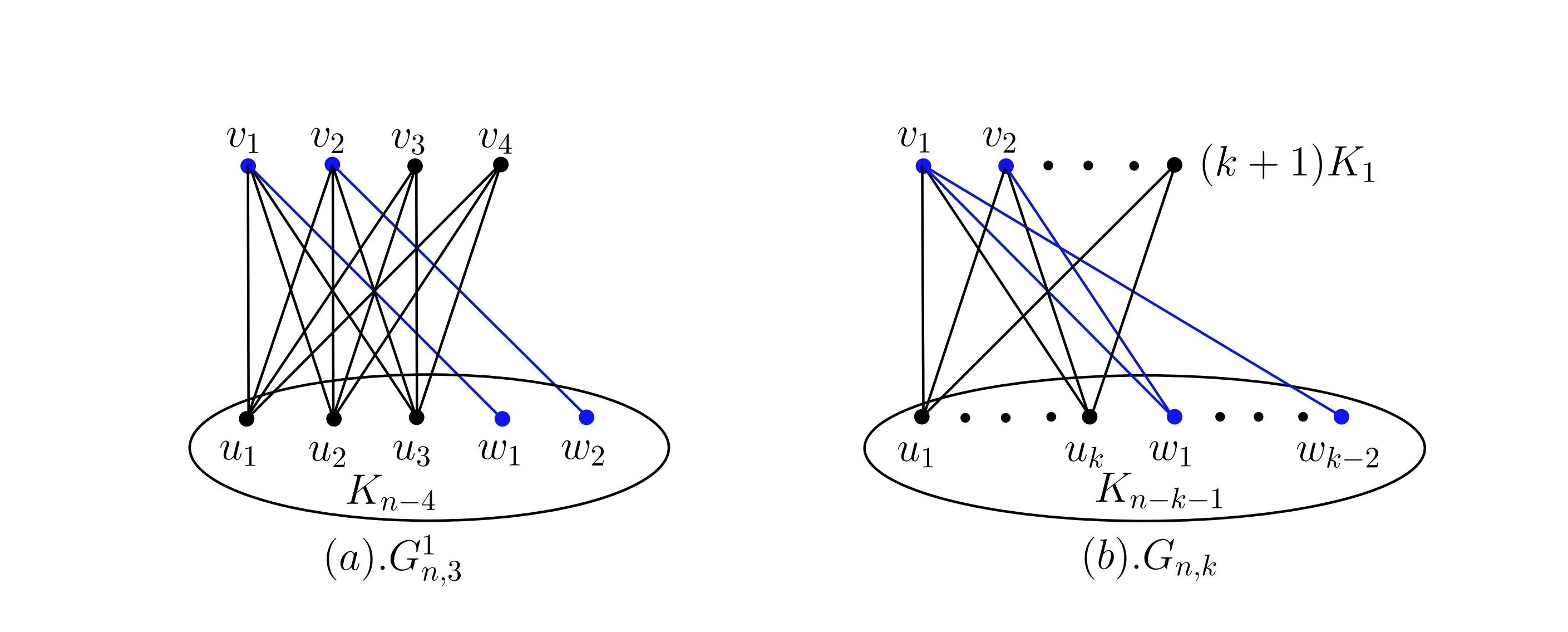}
\\$\mathbf{Fig 1.}$ Graphs $G_{n,k}^1$ and $G_{n,k}^2$
    \label{fig:placeholder}
\end{figure}
Concerning Problem \ref{k-tough}, we prove the following result.
\begin{theorem}\label{tc}
Let $G$ be a 1-tough connected graph of order $n \geq 13k^2-11k-1$ with minimum degree $\delta(G)\ge k$ and $k\ge 3$.
If $$\rho(G)\ge \left\{
  \begin{array}{ll}
    \rho(G^1_{n,3}), ~~ ~~~~~~~~~~~~~~~~~ \hbox{when $k=3$;}\\
    \rho(G_{n,k}), ~~ ~~~~~~~~~~~~~~~~~\hbox{when~$k\ge 4$,}
  \end{array}
\right.$$
then $G$ contains a $k$-factor, unless $G\cong G^1_{n,3}$ or $G\cong G_{n,k}$, see Fig\,1.
\end{theorem}

\section{Preliminaries}
For $X,Y\subseteq V(G)$, we denote by $e_G(X,Y)$ the number of edges in $G$ with one endpoint in $X$ and one endpoint in $Y$.

\begin{lemma}[See \cite{Nosal}]\label{k-factor}
Let $k$ be a positive integer, and let $G$ be a graph. Then $G$ contains a $k$-factor if and only if
$$\delta_G(S, T)=k(|S|-|T|)+\sum_{v\in T}d_G(v)-e_G(S,T)-q_G(S,T)\ge0$$
for all disjoint subsets $S,T\subseteq V(G)$, where $q_G(S,T)$ is the number of the components $C$ of $G-(S\cup T)$ such that
$e_G(V(C),T)+k|V(C)|\equiv 1 ~(mod ~2)$. Moreover, $\delta_G(S,T)\equiv k|V(G)|~ (mod ~2)$.
\end{lemma}

By the well-known Perron-Frobenius theorem (cf. \cite[Section 8.8]{C.Godsil}), we can easily deduce the following result.

\begin{lemma}\label{edge}
If $H$ is a spanning subgraph of a connected graph $G$, then
$\rho(H)\leq\rho(G),$
with equality if and only if $H\cong G$.
\end{lemma}

\begin{lemma}[See \label{degree}\cite{Y.Hong1, V.Nikiforov1}]
Let $G$ be a graph on $n$ vertices and $m$ edges with minimum degree $\delta(G)\ge1$. Then
$$\rho(G)\le\frac{\delta(G)-1}{2}+\sqrt{2e(G)-n\delta(G)+\frac{(\delta(G)+1)^2}{4}},$$
with equality if and only if $G$ is either a $\delta(G)$-regular graph or a bidegreed graph in which each vertex is
of degree either $\delta(G)$ or $n-1$.
\end{lemma}

\begin{lemma} [See \cite{ V.Nikiforov1}]\label{V.N}
For nonnegative integers $p$ and $q$ with $2q\le p(p-1)$ and $0\le x\le p-1$, the function $f(x)=(x-1)/2+\sqrt{2q-px+(x+1)^2/4}$ is decreasing with respect to $x$.
\end{lemma}

\begin{lemma}[See \cite{Nosal}]\label{duoling}
Let $G$ be a connected graph, and let $u,v$ be two vertices of $G$. Suppose that $v_1,v_2,\ldots,v_s\in N_G(v)\backslash N_G(u)$
with $s\ge1$, and $G^*$ is the graph obtained from $G$ by deleting the edges $vv_i$ and adding the edges $uv_i$ for $1\le i\le s$.
Let $x$ be the Perron vector of $A(G)$. If $x_u\ge x_v$, then $\rho(G)<\rho(G^*)$.
\end{lemma}

\begin{lemma} [See \cite{D.Fan1}]\label{Fan}
Let $a$ and $b$ be two positive integers. If $a\ge b\ge3$, then
$$
\binom{a}{2}+\displaystyle\binom{b}{2}<\displaystyle\binom{a+1}{2}+\displaystyle\binom{b-1}{2}.
$$
\end{lemma}

The coefficient recursive bounding method is a technique for proving the non-negativity of polynomials by progressively reducing the polynomial degree via coefficient transformations.
Using this method, we obtain the following result.
\begin{remark}\label{r1}
Let $n\ge 13k^2-11k-1$, $k\ge 5,$ $f_1 = \left(\frac{3k}{2}-\frac{11}{2}\right)n^2 - \left(\frac{9}{2}k^2-16k-\frac{9}{2}\right)n + \frac{9}{4}k^3-10k^2-\frac{33}{4}k>0$
and $f_2 = \frac{k(k-2)}{2}\sqrt{n^2-3(k+1)n+\frac{13}{4}k^2+\frac{15}{2}k+\frac{1}{4}}$.
Then we have
\begin{flalign*}
4(f_1^2\! -\! f_2^2)=&\left(9k^{2}\!-\!66k\!+\!121\right)n^{4}
   \!-\!\left(54k^{3}\!-\!390k^{2}\!+\!650k\!+\!198\right)n^{3}\!+\!(107k^{4}\!-\!793k^{3}\!+\!1202k^{2}\\
   &\!+\!939k\!+\!81)n^{2}\!-\!\left(78k^{5}\!-\!645k^{4}\!+\!905k^{3}\!+\!1413k^{2}\!+\!297k\right)n\! +\!653k^{3}\!+\!272k^{2}\\
   &\!+\!17k^{6}\!-\!181k^{5}\!+\!263k^{4}>0.
\end{flalign*}

\end{remark}

Suppose that $V((k+1)K_1)=\{v_1,\ldots, v_{k+1}\}$, $V(K_{n-2k-1})=\{w_1,\ldots, w_{n-2k-1}\}$.
Let $\mathcal{G}_{n,k}~(k\ge4)$ be the set of 1-tough graphs obtained from $K_k\vee ((k+1)K_1\cup K_{n-2k-1})$ by adding $k-1$ edges between $(k+1)K_1$ and $K_{n-2k-1}$
such that $N_{K_{n-2k-1}}(v_{i})=\{w_{1},\ldots,w_{d_i}\}$, $d_1\ge d_2\ge \cdots\ge d_{i}\ge 0$, where $d_1+\cdots+d_i=k-1$ and $d_1\ge 2, d_2\ge1$, for any $v_i\in V((k+1)K_1)$, $1\le i \le k-2$.
\begin{lemma} \label{main lemma}
Let $G\in \mathcal{G}_{n,k}$. If  $n \geq 13k^2-11k-1$ and $k\ge 4$, then
$\rho(G_{n,k})\ge\rho(G)$, with equality if and only if $G\cong G_{n,k}$, see Fig.\,1. (b).
\end{lemma}
\begin{proof}
Suppose that $T=(k+1)K_1$, $C=K_{n-2k-1}$ and $V(T)=\{v_1,\ldots, v_{k+1}\}$, $V(C)=\{w_1,\ldots, w_{n-2k-1}\}$, $V(K_k)=\{u_1,\ldots, u_k\}$.
Let $G\in \mathcal{G}_{n,k}$,
by the definition of $\mathcal{G}_{n,k}$, we can assume that $d_{C}(v_i)=d_i~(\ge0)$, and $d_1\ge d_2\ge\cdots\ge d_{i}\ge0$, $d_1\ge 2, d_2\ge1$ for $v_i\in V(T)$, $1\le i \le k-2$.
Let $d_{T}(w_i)=r_i~(\ge0)$, then we have $N_{T}(w_{i})=\{v_{1},\ldots,v_{r_i}\}$, where $r_1\ge r_2\ge\cdots\ge r_{i}\ge0$ and $r_1\ge 2, r_2\ge1$ for $w_i\in V(C)$, $1\le i\le k-2$.
Note that $\sum_{i=1}^{d_i}=\sum_{i=1}^{r_i}=k-1$. Obviously, $\mathcal{G}_{n,4}=\{G_{n,4}\}$ when $k=4$ and
\begin{align*}
    G_{n,k} = G + \sum_{i=d_{1}+1}^{k-2} v_{1} w_{i} - \sum_{i=2}^{d_{2}} v_{2} w_{i} - \sum_{i=3}^{r_{1}} \sum_{j=1}^{d_{r_i}} v_{i} w_{j}, \text{when}~ k\ge 5.
\end{align*}
Let $X=(x_1,\dots, x_{w_n})^T$ be the Perron vector of $A(G)$. By symmetry, we can suppose that
$x_{u_{1}} = \cdots = x_{u_{k}} = x_{1},$ $x_{w_{d_1+1}} = \cdots = x_{w_{n}} = x_{2},$ $x_{v_{r_1+1}} = \cdots = x_{v_{k+1}} = x_{3}.$
Clearly, $x_{1} > x_{w_{1}} \geq x_{w_{2}} \geq \cdots \geq x_{w_{d_1}} > x_{2} > x_{v_{1}} \geq x_{v_2}\ge \dots \ge x_{v_{r_1}}>x_3$ by eigenvalue equation $\rho(G) x_i=A(G)x_{x_i}$ for any vertex in $V(G)$.
Let $Y = (y_1, \dots, y_{n})^T$ be the Perron vector of $A(G_{n,k})$. By symmetry, we can suppose that $y_{v_{1}} = y_{1},$ $y_{v_{2}} = y_{2},$ $y_{v_{3}} = \cdots = y_{v_{k+1}} = y_{3},$ $y_{u_{1}} = \cdots = y_{u_{k}} = y_{4},$ $y_{w_{1}} = y_{5},$  $y_{w_{2}} = \cdots = y_{w_{k-2}} = y_{6},$ $y_{w_{k-1}} = \cdots = y_{w_{n}} = y_{7}.$ Similarly, we can obtain
   $ y_{4} > y_{5} > y_{6} > y_{7} > y_{1} > y_{2} > y_{3}.$
Let $\rho(G)=\rho_1$ , $\rho(G_{n,k})=\rho_2$.
Then we have
\begin{align*}
&Y^{T}(\rho_{2} - \rho_{1})X = Y^{T}\bigl(A(G_{n,k}) - A(G)\bigr)X \\
&= \sum_{i=d_{1}+1}^{k-2} (x_{v_{1}}y_{w_{i}} + x_{w_{i}}y_{v_{1}}) - \sum_{i=2}^{d_{2}} (x_{v_{2}}y_{w_{i}} + x_{w_{i}}y_{v_{2}})
- \sum_{i=3}^{r_{1}} \sum_{j=1}^{d_{3}} (x_{v_{i}}y_{w_{j}} + x_{w_{j}}y_{v_{i}}) \\
&> \!(k\! -\! d_{1}\! -\! 2)(x_{v_{1}} y_{6}\! +\! x_{2} y_{1} \!-\! x_{v_{3}} y_{5} \!- \!x_{w_{1}} y_{3})\!+\!(\!d_{2}\! -\! 1)(x_{v_{2}}y_{6}\! +\! x_{w_{2}}y_{2}\!-\!x_{v_{3}}y_{5}\! -\! x_{w_{1}}y_{3})\\
&~~~~~\text{(since } x_{w_2}\ge x_{w_i}, x_{v_3}\ge x_{v_j}, 1\le i\le d_2, 3\le j\le r_1  \text{)}.
\end{align*}
From the eigenvalue equations
$A(G)X= \rho_1 X$,
$A(G_{n,k})Y = \rho_2 Y,$ we have
\begin{numcases}{}
\rho_{2}y_{1} = ky_{4}+y_{5}+(k-3)y_{6}~~~~~~~~~~~~~~~~~~~~~~~~~~~~~~~~~~~~~~~~~~~~~~~~~~~~~~~~~~~~~~~~~~\label{eq:1} \\
\rho_{2}y_{3} = ky_{4}\label{eq:2}\\
\rho_{2}y_4=y_1+y_2+ky_3+(k-1)y_4+y_5+(k-3)y_6+(n-3k+1)y_7\label{eq:0.3}\\
\rho_{2}y_{5} = y_{1}+y_{2}+ky_{4}+(k-3)y_{6}+(n-3k+1)y_{7} \label{eq:3}\\
\rho_{2}y_{6} = y_{1}+ky_{4}+y_{5}+(k-4)y_{6}+(n-3k+1)y_{7}\label{eq:4} \\
\rho _{1}x_{w_{1}} = \sum_{i=1}^{r_{1}}x_{v_{i}}+k x_{1}+\sum_{i=2}^{d_{1}}x_{w_{i}}+(n-2k-d_{1}-1)x_{2}\label{eq:5} \\
\rho _{1}x_{2} = k x_{1}+\sum_{i=1}^{d_{1}}x_{w_{i}}+(n-2k-d_{1}-2)x_{2}\label{eq:6}\\
\rho_{2}y_{2} = ky_{4}+y_{5}\label{eq:13} \\
\rho _{1}x_{w_{2}} = \sum_{i=1}^{r_{2}}x_{v_{i}}+k x_{1}+x_{w_{1}}+\sum_{i=3}^{d_{1}}x_{w_{i}}+(n-2k-d_{1}-1)x_{2}~~~~~~~~~~~~~~~~~~~~~\label{eq:14}
\end{numcases}
Next, we prove the following two facts.
{\flushleft\bf Fact 1.}  $ x_{v_{1}} y_{6} + x_{2} y_{1} - x_{v_{3}} y_{5} - x_{w_{1}} y_{3} > 0.$
{\flushleft\bf \emph{Proof}.}
Note that
\begin{numcases}{}
y_1 = y_3 + \frac{y_5 + (k-3)y_6}{\rho_2}~~~\text{(by } (\ref{eq:1}), (\ref{eq:2})\text{)}~~~~~~~~~~~~~~~~~~~~~~~~~~~~~~~~~~~~~~~~~~~~~~~~~~\label{eq:7} \\
y_5 = y_3 + \frac{y_1 + y_2 + (k-3)y_6 + (n-3k+1)y_7}{\rho_2}~~~\text{(by } (\ref{eq:2}), (\ref{eq:3})\text{)}\label{eq:8}\\
y_6 = y_3 + \frac{y_1 + y_5 + (k-4)y_6 + (n-3k+1)y_7}{\rho_2}~~~\text{(by } (\ref{eq:2}), (\ref{eq:4})\text{)}\label{eq:9}\\
(x_2- x_{w_1}) y_3  = -\frac{\sum_{i=1}^{r_1} x_{v_i}}{\rho_1+1} y_3~~~\text{(by } (\ref{eq:5}), (\ref{eq:6})\text{)}\label{eq:11}
\end{numcases}
Combining (\ref{eq:7})-(\ref{eq:9}) and $x_{v_1} \geq x_{v_3},$
then we have
\begin{align}
&(\rho_1+1)\rho_2 (x_{v_1}y_6 + x_2y_1 - x_{v_3}y_5 - x_{w_1}y_3)\nonumber\\
 =&(\rho_1+1)\rho_2 \Big( x_{v_1}y_{3} - x_{v_3}y_{3} + x_{2}y_{3} - x_{w_1}y_{3}  + \frac{1}{\rho_{2}}\left((x_{v_1} - x_{v_3})y_{1} - x_{v_3}y_{2}+(x_{v_1} + x_{2})y_{5}\right)\nonumber \\
& + \frac{1}{\rho_{2}}\left((k-4)x_{v_1} + (k-3)(x_{2} -x_{v_3})\right) y_{6}
   + (n-3k+1)(x_{v_1} - x_{v_3})y_{7}\Big) \nonumber\\
\ge&  (\rho_1+1)\big(\left((k-4)x_{v_1} + (k-3)(x_{2}-x_{v_3})\right) y_{6}
   + (n-3k+1)(x_{v_1} - x_{v_3})y_{7}\nonumber\\
   & +(x_{v_1} - x_{v_3})y_{1} - x_{v_3}y_{2}+(x_{v_1} + x_{2})y_{5}\big)-\sum_{i=1}^{r_1}\rho_2 x_{v_i} y_3\quad\text{(by } (\ref{eq:11}))\nonumber\\
\ge& (\rho_1\!+\!1) y_5 x_2\! +\! (\rho_1\!+\!1)(k\!-\!3)(y_6 x_2\!-\!y_6x_{v_3})\! +\! (\rho_1\!+\!1) y_1 (x_{v_1}\!-\!x_{v_3})\nonumber
\!+\! (\rho_1\!+\!1)( y_5 x_{v_1}\!-\!y_2 x_{v_3}) \nonumber\\
&\!+\! (\rho_1\!+\!1)(k\!-\!4)y_6 x_{v_1}\!+\! (\rho_1\!+\!1)(n\!-\!3k\!+\!1)y_7(x_{v_1}\!-\!x_{v_3})\!-\!\sum_{i=1}^{r_1} x_{v_i} k y_4
~\text{(by }\!\rho_2 y_3\! =\! k y_4\! ~\!(see (\ref{eq:2}) \text{)}\nonumber\\
>&(\rho_1+1) y_5 x_2 + (\rho_1+1)(k-4)y_6 x_{v_1} - \sum_{i=1}^{r_1} x_{v_i} k y_4\label{eq:0.14}\\
&\text{(by }x_2 > x_{v_3},  x_{v_1} \geq x_{v_3}, y_5 > y_2\text{)}\nonumber.
\end{align}
Since $\sum_{i=1}^{r_1} x_{v_i} \leq r_1 x_{v_1}$, $r_1 \leq k-2$ and $G$ contains $K_{n-k-1}$ as a proper subgraph, we have $\rho_1+1 > n-k-1 \geq 13k^2-12k-2$, $n-3k+1 \geq 13k^2-14k$.
Through computation, we have $(\rho_1+1)(k-4)>k(k-2),$ $k(\rho_1+1)>k(k-1)(k-2)$ and $(\rho_1+1)(k-4)^2>k(k-2)(k-3)$.
Combining the values of $y_4$, $y_5$, $y_6$ (see (\ref{eq:0.3})-(\ref{eq:4})), $y_7>y_3,$ and $y_4>y_1>y_2,$
we obtain
\begin{align*}
&\rho_2\left((\rho_1+1)y_5 x_2 + (\rho_1+1)(k-4)y_6 x_{v_1} - \sum_{i=1}^{r_1} x_{v_i} k y_4\right) \notag\\
&\ge (\rho_{1}\!+\!1)x_{2}(y_{1}\!+\!y_{2}\!+\!(k\!-\!3)y_{6}) \!+\!\left((\rho_{1}\!+\!1)(k\!-\!4)\!-\!k(k\!-\!2)\right)(y_1\!+\!y_2\!+\!y_5\!+\!(n\!-\!3k\!+\!1)y_7)x_{v_1}\\
&\quad\!+\!\left(k(\rho_1\!+\!1)\!-\!k(k\!-\!1)(k\!-\!2)\right)(y_4\!+\!y_3)x_{v_1}+\left((\rho_{1}\!+\!1)(k\!-\!4)^2-k(k-2)(k-3)\right)y_6x_{v_1}\\
&>(\rho_{1}+1)x_{2}(y_{1}+y_{2}+(k-3)y_{6})\\
&>0,
\end{align*}
that is $$(\rho_1+1)y_5 x_2 + (\rho_1+1)(k-4)y_6 x_{v_1} - \sum_{i=1}^{r_1} x_{v_i} k y_4>0.$$
Putting this into (\ref{eq:0.14}), we get $ x_{v_{1}} y_{6} + x_{2} y_{1} - x_{v_{3}} y_{5} - x_{w_{1}} y_{3} > 0$, as desired.

{\flushleft\bf Fact 2.}  $ x_{v_{1}} y_{6} + x_{2} y_{1} - x_{v_{2}} y_{6} - x_{w_{2}} y_{2} > 0.$
{\flushleft\bf \emph{Proof}.}
Note that
\begin{numcases}{}
y_2 = y_1 - \frac{(k-3)y_6}{\rho_2}~~~~~~~~\text{(by } (\ref{eq:1}), (\ref{eq:13})\text{)}~~~~~~~~~~~~~~~~~~~~~~~~~~~~~~~~~~~~~~\label{eq:15}\\
 x_2 = x_{w_2} - \frac{\sum_{i=1}^{r_2} x_{v_i}}{\rho_1+1}~~~~~~\text{(by } (\ref{eq:6}), (\ref{eq:14})\text{)}~~~~~~~~~~~~~~~~~~~~~~~~~~~~~~~~~~~~~~~~~~\label{eq:16}\\
  y_6 = y_1 + \frac{n-3k+1}{\rho_2+1} y_7~~~\text{(by } (\ref{eq:1}), (\ref{eq:4})\text{)}~~~~~~~~~~~~~~~~~~~~~~~~~~~~~~~~~~~~~~~~~~\label{eq:17}
\end{numcases}
Since $x_{v_1}y_6 - x_{v_2}y_6 \geq 0$ by $x_{v_1} \geq x_{v_2}$, $2r_2\le r_1+r_2\le k-1$, we have $r_2 \leq \frac{k-1}{2}$, that is,
$\rho_2 \sum_{i=1}^{r_2} x_{v_i} \leq \rho_2 r_2 x_{v_1} \leq \frac{k-1}{2}\rho_2 x_{v_1}. $
By Lemmas \ref{degree} and \ref{V.N}, we have
$$\rho_2 \leq \frac{k-1}{2} + \sqrt{n^2-3(k+1)n+\frac{13}{4}k^2+\frac{15}{2}k+\frac{1}{4}}.$$
It is implies that $\rho_2(\rho_2+1) \leq n^2-3(k+1)n+\frac{7}{2}k^2+\frac{15}{2}k+k\sqrt{n^2-3(k+1)n+\frac{13}{4}k^2+\frac{15}{2}k+\frac{1}{4}}$. Then we have
\begin{align*}
 &x_{v_1}y_{6} + x_{2}y_{1} - x_{v_2}y_{6} - x_{w_{2}}y_{2} \\
 &\geq x_{2}y_{1} - x_{w_{2}}y_{2} \\
& = (x_{2} - x_{w_{2}})y_{1} + \frac{(k-3)}{\rho_{2}} x_{w_{2}} y_{6} \quad (\text{by (\ref{eq:15})}) \\
&=\frac{1}{\rho_2(\rho_1+1)}\left((\rho_1+1)(k-3)x_{w_2}y_6 - \sum_{i=1}^{r_2} \rho_2 x_{v_i} y_1\right)\quad (\text{by (\ref{eq:16})})\\
&> \frac{1}{\rho_2(\rho_1+1)}\left(\rho_1(k-3)x_{w_2}y_6 - \sum_{i=1}^{r_2} \rho_2 x_{v_i} y_1\right)\\
&\!>\! \frac{1}{\rho_2(\rho_1\!+\!1)}\left(\rho_1(k\!-\!3)x_{w_2}\left(1\! +\! \frac{n\!-\!3k\!+\!1}{\rho_2\!+\!1}\right)y_1\! -\! \rho_2 \sum_{i=1}^{r_2} x_{v_i} y_1\right) \quad (\text{by } y_7\! >\! y_1,(\ref{eq:17})) \notag\\
&\!= \!\frac{(k\!-\!3)(n\!-\!3k\!+\!\rho_2\!+\!2)}{\rho_2(\rho_1+1)(\rho_2\!+\!1)}\left(\sum_{i\!=\!1}^{r_2} x_{v_i}\! +\! kx_1\! + \!x_{w_1}\! +\!\sum_{i=3}^{d_1} x_{w_i}\! +\! (n\!-\!2k\!-\!d_1\!-\!1)x_2\right)y_1 \\
&\quad - \frac{\rho_2}{\rho_2(\rho_1+1)} \sum_{i=1}^{r_2} x_{v_i} y_1 \quad (\text{by } (\ref{eq:14}))\\
&>\frac{x_{v_1}y_1}{\rho_2(\rho_1+1)}\left(\frac{1}{\rho_2+1}(k-3)(n-k+1)(2n-4k)-\frac{k-1}{2}\rho_2\right)\\
&~~~(\text{by } x_1 > x_{w_i} > x_2 > x_{v_1}\ge x_{v_j},\rho_2>n-k-2, 1\le i\le d_1, 1\le j\le r_2)\\
&\ge  \frac{x_{v_1}y_1}{\rho_2(\rho_1+1)(\rho_2+1)}\Big(\Big(\frac{3k}{2}-\frac{11}{2}\Big)n^2 - \Big(\frac{9}{2}k^2-16k-\frac{9}{2}\Big)n + \frac{9}{4}k^3-10k^2-\frac{33}{4}k \notag\\
&\quad - \frac{k(k-2)}{2}\sqrt{n^2-3(k+1)n+\frac{13}{4}k^2+\frac{15}{2}k+\frac{1}{4}}\Big)\\
&= \frac{x_{v_1}y_1(f_1-f_2)}{\rho_2(\rho_1+1)(\rho_2+1)}>0\quad (\text{by Remark } (1 )),~ \text{as desired}.
\end{align*}
By Facts 1 and 2, we can obtain that $\rho(G)\le \rho(G_{n,k})$, with equality if and only if $G\cong G_{n,k}$.
This completes the proof of Lemma \ref{main lemma}.
\end{proof}

\section{Proof of Theorem \ref{tc}}
For any $S\subseteq G$, let $G[S]$ be the subgraph of $G$ induced by $S$ and $e(S)$ be the number of edges in $G[S]$.
Let $\mathcal{G}_n^k$ be the set of all connected 1-tough graphs of order $n \geq 13k^2-11k-1$ with $\delta(G) \ge k$ and no $k$-factor.
\renewcommand\proofname{\bf Proof of Theorem \ref{tc}}
\begin{proof}
Suppose that $G^*\in \mathcal{G}_n^k$ with the maximum spectral radius. By Lemma \ref{k-factor}, there exist two disjoint subsets $S,T\subseteq V(G^*)$, satisfying
$|S\cup T|$ as large as possible such that
\begin{equation}\label{equ::delta(S,T)}
\begin{aligned}
\delta_{G^*}(S, T)=k(|S|-|T|)+\sum_{v\in T}d_{G^*}(v)-e_{G^*}(S,T)-q_{G^*}(S,T)\le -2,
\end{aligned}
\end{equation}
where $q_{G^*}(S,T)$ is the number of the components $C$ of $G^*-(S\cup T)$ such that
$e_{G^*}(C,T)+k|V(C)|\equiv 1 ~(mod ~2)$. Let $C_1, C_2, \ldots, C_q$ be the
components of $G^*-(S\cup T)$ such that $e_{G^*}(C_i,T)+k|V(C)|\equiv 1 ~(mod ~2)$ and $|V(C_i)|=c_i$ where $1\le i\le q$.
Then we have $s+t\ge q$ due to $G$ is a connected 1-tough graph.
Let $|S|=s$, $|T|=t$ and $q_{G^*}(S,T)=q$.  By (\ref{equ::delta(S,T)}), we obtain
\begin{equation}\label{dG(v)}
\begin{aligned}
\sum_{v\in T}d_{G^*}(v)\le k(t-s)+e_{G^*}(S,T)+q_{G^*}(S,T)-2.
\end{aligned}
\end{equation}
Together with Lemma \ref{edge} and the choice of $G^*$, we get $G^*[S\cup T]=K_{s,t}$
and $G^*-T\cong K_s\vee (K_{c_1}\cup \cdots\cup K_{c_q}\cup K_{c_{q+1}})$, where $c_{q+1}=n-t-s-\sum_{i=1}^{q}c_i.$
Now,
we divide the proof into the following eight claims.

{\flushleft\bf Claim 1.} $t\ge s+1$.

{\flushleft\bf \emph{Proof}.} Otherwise, $s\ge t$. If $q=0$, then $\delta_{G^*}(S, T)=k(s-t)+\sum_{v\in T}d_{G^*-S}(v)\ge 0$, which
contradicts (\ref{equ::delta(S,T)}). If $q\ge 1$, then we have
$$\delta_{G^*}(S,T)=k(s-t)+\sum_{v\in T}d_{G^*-S}(v)-q\ge k(s-t)+\sum_{i=1}^{q}e_{G^*}(C_i, T)-q\ge 0,$$
which also contradicts (\ref{equ::delta(S,T)}).
$\Box$

{\flushleft\bf Claim 2.} $e(G^*)> \displaystyle\binom{n-k-1}{2}+k+2$.

{\flushleft\bf \emph{Proof}.} Note that $G_{n,k}\in \mathcal{G}_n^k$ and contains a $K_{n-k-1}$ as a proper subgraph. Then we have
$\rho (G^*)\ge \rho (G_{n,k})>\rho(K_{n-k-1})=n-k-2$ by Lemma \ref{edge}.
Combining this with Lemmas \ref{degree} and \ref{V.N}, we obtain
$$n-k-2<\rho(G^*)\le\frac{k-1}{2}+\sqrt{2e(G^*)-nk+\frac{(k+1)^2}{4}}.$$
The claim follows immediately.
$\Box$

{\flushleft\bf Claim 3.} If $q\ge 1$, then $c_i\ge k$, where $1\le i\le q$.

{\flushleft\bf \emph{Proof}.} Otherwise, there exists some $C_j$ such that $c_j=k-a$, where $1\le a\le k-1$ and
\begin{equation}\label{e(Cj,T)}
\begin{aligned}
e_{G^*}(C_j,T)+kc_j\equiv 1 ~(mod ~2).
\end{aligned}
\end{equation}
Then we divide the proof into the following two cases.

{\flushleft\bf Case 1.} $k$ is even.

Then we have $k(k-a)-1$ is odd and $e_{G^*}(C_j,T)$ is odd by (\ref{e(Cj,T)}). If $e_{G^*}(C_j,T)\le k(k-a)-1$.
Then let $T'=T\cup V(C_j)$ and so $|T'|=t'=t+c_j= t+k-a$. We have
\begin{equation*}
\begin{aligned}
\delta_{G^*}(S,T')&=k(s-t')+\sum_{v\in T}d_{G^*}(v)+\sum_{v\in V(C_j)}d_{G^*}(v)-\left(e_{G^*}(S,T)+e_{G^*}(S,C_j)\right)-(q-1)\\
&=k(s-t)+\sum_{v\in T}d_{G^*}(v)-e_{G^*}(S,T)-q-k(k-a)+e_{G^*}(C_j,T)+1\\
&=\delta_{G^*}(S,T)-\left(k(k-a)-1-e_{G^*}(C_j,T)\right)\\
&\le\delta_{G^*}(S,T)~~\left(\mbox{since $e_{G^*}(C_j,T)\le k(k-a)-1$}\right),
\end{aligned}
\end{equation*}
which contradicts the maximality of $S\cup T$.
If $e_{G^*}(C_j,T)\ge k(k-a)+1$.
Then let $S'=S\cup V(C_j)$ and so $|S'|=s'=s+c_j= s+k-a$. We have
\begin{equation*}
\begin{aligned}
\delta_{G^*}(S',T)&=k(s'-t)+\sum_{v\in T}d_{G^*}(v)-(e_{G^*}(S,T)+e_{G^*}(C_j,T))-(q-1)\\
&=k(s-t)+\sum_{v\in T}d_{G^*}(v)-e_{G^*}(S,T)-q+k(k-a)-e_{G^*}(C_j,T)+1\\
&=\delta_{G^*}(S,T)-(e_{G^*}(C_j,T)-k(k-a)-1)\\
&\le\delta_{G^*}(S,T)~~(\mbox{since $e_{G^*}(C_j,T)\ge k(k-a)+1$}),
\end{aligned}
\end{equation*}
which also contradicts the maximality of $S\cup T$.

{\flushleft\bf Case 2.} $k$ is odd.

If $k-a$ is odd, then $k(k-a)-1$ is even, and hence $e_{G^*}(C_j,T))$ is even by (\ref{e(Cj,T)}).
If $k-a$ is even, then $k(k-a)-1$ is odd, and hence $e_{G^*}(C_j,T))$ is odd by (\ref{e(Cj,T)}).
Therefore, we may distinguish cases based on the values of $e_{G^*}(C_j,T))$ and $k(k-a)-1$.
Using the same argument as in Case 1, we can find a larger set $T'$ or $S'$ such that
$\delta_{G^*}(S,T')\le -2$ or $\delta_{G^*}(S',T)\le -2$, which contradicts the maximality of $S\cup T$.

This completes the proof of claim 3.
$\Box$

{\flushleft\bf Claim 4.} $q\le 1$.

{\flushleft\bf \emph{Proof}.} Otherwise, $q\ge 2$. We can obtain that $c_i\ge k, 1\le i\le q$ by Claim 3 and so $n\ge qk+s+t$. Recall that $G^*[S\cup T]=K_{s,t}$
and $G^*-T\cong K_s\vee (K_{c_1}\cup \cdots\cup K_{c_q}\cup K_{c_{q+1}})$, where $c_{q+1}=n-t-s-\sum_{i=1}^{q}c_i.$ \
Then by Lemma \ref{Fan} and inequality (\ref{dG(v)}), we have
\begin{align}
e(G^*)&\le\sum_{v\in T}d_{G^*}(v)+\sum_{i=1}^{q-1}\displaystyle\binom{c_i}{2}+\displaystyle\binom{n-s-t-\sum_{i=1}^{q-1}c_i}{2}+\displaystyle\binom{s}{2}+s(n-s-t)\nonumber\\
&\le k(t\!-\!s)\!+\!st\!+\!q\!-\!2\!+\!(q\!-\!1)\displaystyle\binom{k}{2}\!+\!\displaystyle\binom{n\!-\!t\!-\!s\!-\!(q\!-\!1)k}{2}+\displaystyle\binom{s}{2}\!+\!s(n\!-\!s\!-\!t)\label{e(G)}.
\end{align}
We divide the proof into the following four cases.

{\flushleft\bf Case 1.} $1\le t\le k$ and  $q\ge3$.
Then by  Claims 1-2 and inequality (\ref{e(G)}), we have
\begin{equation*}
\begin{aligned}
0&>\displaystyle\binom{n-k-1}{2}+k+2-e(G^*)\\
&\ge(kq-2k+t-1)n+\frac{(1\!+\!q\!-\!q^2)k^2}{2} -\frac{(2q(s\!+\!t)\!-\!4s\!-\!5)k}{2}-\frac{(1\!+\!t\!+\!2s)t}{2}+5-q\\
&\ge \frac{t^2\!+\!(2kq\!-\!4k\!-\!3)t}{2}+\frac{k^2q^2\!-\!(3k^2\!+\!2k\!+\!2)q}{2}+\frac{k^2\!+\!5k}{2}+5-s ~~~(\mbox{since $n\ge qk+s+t$})\\
&\ge \frac{t^2\!+\!(2kq\!-\!4k\!-\!5)t}{2}+\frac{k^2q^2\!-\!(3k^2\!+\!2k\!+\!2)q}{2}+\frac{k^2\!+\!5k}{2}+6~~~(\mbox{since $s\le t-1$})\\
&\ge \frac{t^2\!+\!(2k\!-\!5)t}{2}+\frac{k^2\!-\!2k}{2}+3~~~(\mbox{since $q\ge 3$})\\
&\ge\frac{k^2+k+2}{2}~~~(\mbox{since $t\ge 1$})\\
&>0 ~~~(\mbox{since $k\ge 1$}),~ \mbox{a contradiction.}
\end{aligned}
\end{equation*}

{\flushleft\bf Case 2.} $1\le t\le k$ and $q=2$.

Recall that $n\ge 13k^2-11k-1$. If $t\ge 2$. Then by Claims 1-2 and inequality (\ref{e(G)}), we have
\begin{equation*}
\begin{aligned}
0&>\displaystyle\binom{n-k-1}{2}+k+2-e(G^*)\\
&\ge(t-1)n-\frac{k^2+t^2+4kt+t-5k}{2} +3- st~~~(\mbox{Since inequality (\ref{e(G)})})\\
&\ge (t-1)n-\frac{k^2+3t^2+4kt-t-5k}{2}+3 ~~~(\mbox{since $s\le t-1$})\\
&\ge n-4k^2+ \frac{5}{2}k+4~~~(\mbox{since $2\le t\le k$})\\
&>0~(\mbox{since $n\ge 13k^2-11k-1$}), ~ \mbox{a contradiction.}
\end{aligned}
\end{equation*}
Let $t=1$ and $q=2$, then $s=0$ by Claim 1. We assert that $c_i\ge k+1$ for $i=1,2$.
Otherwise, there exists some $C_i$ such that $c_i=k$ by Claim 3, where $i=1,2$.
Then we have $e_{G^*}(C_i,T)=k$ as $\delta(G^*)\ge k$. Which contradicts of
$e_{G^*}(C_i,T)+kc_i\equiv 1 ~(mod ~2)$.
Therefore, by inequality (\ref{e(G)}), we can get $e(G^*)\le k+\frac{k(k+1)}{2}+\displaystyle\binom{n-k-2}{2}$.
Combining this with Claim 2 and $n\ge 13k^2-11k-1$, we have
$$0>\displaystyle\binom{n-k-1}{2}+k+2-e(G^*)\ge n-\frac{k^2}{2}-\frac{3k}{2}>0, ~ \mbox{a contradiction.}$$
{\flushleft\bf Case 3.} $t\ge k+1$ and $q\ge3$.
\enlargethispage*{5\baselineskip}
It is not hard to see that at least $\sum_{i=1}^{q-1}\sum_{j=i+1}^qc_ic_j$ edges here are not in $G$. Without loss
of generality, assume $c_q\ge\cdots\ge c_1\ge k$. Then we have
$$\sum_{i=1}^{q-1}\sum_{j=i+1}^qc_ic_j\ge c_1((q-1)c_1+(q-2)c_1+\cdots+c_1)\ge \frac{q(q-1)k^2}{2}.$$
Combining this and inequality (\ref{dG(v)}), we have
\begin{equation*}
\begin{aligned}
e(G^*)&\le\sum_{v\in T}d_{G^*}(v)\!+\displaystyle\binom{ n\!-\!t}{2}\!-\!\sum_{i=1}^{q\!-\!1}\sum_{j=i\!+\!1}^qc_ic_j\\
&\le k(t\!-\!s)\!+\!st\!+\!q\!-\!2\!+\!\frac{(n\!-\!t)(n\!-\!t\!-\!1)}{2}\!-\!\frac{k^2(q\!-\!1)q}{2}.
\end{aligned}
\end{equation*}
Hence, by Claim 2, we obtain that
\begin{equation*}
\begin{aligned}
0&>\displaystyle\binom{n-k-1}{2}+k+2-e(G^*)\\
&\ge(t-k-1)n+\frac{k^2(q^2-q+1)-t(t+2k-1)+5k}{2}+5-q-(t-k)s\\
&\ge \frac{k^2q^2-(3k^2-2kt+2k+2)q+(t-4k-5)t+k^2+5k}{2}+6\\
&~~~~~(\mbox{since $n\ge kq+s+t$ and $s\le t-1$})\\
&\ge \frac{k^2+(2t-1)k+t^2-5t+6}{2}~~~(\mbox{since $q\ge 3$})\\
&\ge 2k^2-k+1 ~~~(\mbox{since $t\ge k+1$})\\
&>0 ~~~(\mbox{since $k\ge 1$}), ~ \mbox{a contradiction.}
\end{aligned}
\end{equation*}

{\flushleft\bf Case 4.} $t\ge k+1$ and  $q=2$.

If $t\ge k+2$, then by Claim 2 and inequality (\ref{e(G)}), we have
\begin{equation*}
\begin{aligned}
0&>\displaystyle\binom{n-k-1}{2}+k+2-e(G^*)\\
&\ge(t-1)n-\frac{t^2+k^2+4kt+t-5k}{2}+3-st~~~(\mbox{Since inequality (\ref{e(G)})})\\
&\ge \frac{t^2-k^2+k-3t}{2}+3-s~~~(\mbox{since $n\ge 2k+s+t$})\\
&\ge \frac{t^2-k^2+k-5t}{2}+4~~~(\mbox{since $s\le t-1$})\\
&\ge1~~~(\mbox{since $t\ge k+2$}), ~ \mbox{a contradiction.}
\end{aligned}
\end{equation*}
If $t= k+1$, then by Claim 2 inequality (\ref{e(G)}) we have
\begin{equation*}
\begin{aligned}
0&>\displaystyle\binom{n-k-1}{2}+k+2-e(G^*)\\
&\ge nk-3k^2-k+2- s(k+1) ~~~(\mbox{Since inequality (\ref{e(G)})})\\
&\ge nk-4k^2-2k+2~~~(\mbox{since $s\le t-1=k$})\\
&>0 ~~~(\mbox{since $n\ge 13k^2-11k-1$}),~ \mbox{a contradiction.}
\end{aligned}
\end{equation*}
Thus, it is completes the proof of claim 4.
$\Box$

{\flushleft\bf Claim 5.} $t\ge k+1$.

{\flushleft\bf \emph{Proof}.}
By Claim 4, inequality (\ref{dG(v)}) and $\delta(G)\ge k$, we have
\begin{equation*}\label{d(T)}
\begin{aligned}
kt\le \sum_{v\in T}d_{G^*}(v)\le k(t-s)+st-1,
\end{aligned}
\end{equation*}
and thus $t\ge k+\frac{1}{s}$. It is implies that $t\ge k+1$ because $t$ is a positive integer.
$\Box$

{\flushleft\bf Claim 6.} $n\ge s+t+k+1$.
{\flushleft\bf \emph{Proof}.}
Otherwise, let $n=s+t+b$, where $1\le b\le k$. Then we have $n\le 2t-1+b$ as $s\le t-1$.
Recall that  $n\ge 13k^2-11k-1$, so we get $t\ge \frac{13}{2}k^2-\frac{11}{2}k-\frac{b}{2}\ge\frac{13}{2}k^2-6k$.
Note that
\begin{equation}\label{e(G)2}
\begin{aligned}
e(G^*)\le \sum_{v\in T}d_{G^*}(v)+\displaystyle\binom{ n-t}{2}\le k(t-s)+st-1+\frac{(n-t)(n-t-1)}{2}.
\end{aligned}
\end{equation}
If $k\ge 2,$ then by Claim 2, we have
\begin{equation*}
\begin{aligned}
0&>\displaystyle\binom{n-k-1}{2}+k+2-e(G^*)\\
&\ge (t-k-1)n+\frac{k^2+5k-t^2-t}{2}+4-(t-k)s-kt\\
&=\frac{k^2+5k+t^2-3t}{2}+(t-k-1)b+4-2kt-s~~~(\mbox{since $n=s+t+b$})\\
&\ge \frac{(t-4k-5)t+k^2+5k}{2}+5~~~(\mbox{since $s\le t-1, t\ge k+1$})\\
&\ge\frac{t+k^2+5k}{2}+5~~~\left(\mbox{since $t\ge \frac{13}{2}k^2-6k, k\ge2$}  \right)\\
&>0, ~ \mbox{a contradiction.}
\end{aligned}
\end{equation*}
Let $k=1$, then $t\ge2$ and $b=1$. We have
\begin{equation*}
\begin{aligned}
0&>\displaystyle\binom{n-k-1}{2}+k+2-e(G^*)\\
&\ge(t-2)n-\frac{t^3+3t}{2}+7-(t-1)s~~~(\mbox{since $k=1$})\\
&=\frac{t^2-5t}{2}+5-s~~~(\mbox{since $n=s+t+1$})\\
&\ge\frac{(t-3)(t-4)}{2} ~~~(\mbox{since $s\le t-1$}) \\
&\ge0, ~ \mbox{a contradiction.}
\end{aligned}
\end{equation*}
Thus, the claim hold.
$\Box$

{\flushleft\bf Claim 7.} $t= k+1$.
{\flushleft\bf \emph{Proof}.}
Otherwise, $t\ge k+2$. If $t\le 3k-1$, then by Claim 2, we have
\begin{equation*}
\begin{aligned}
0&>\displaystyle\binom{n-k-1}{2}+k+2-e(G^*)\\
&\ge (t-k-1)n+\frac{k^2+5k-t^2-t}{2}+4-(t-k)s-kt ~~~(\mbox{Since inequality (\ref{e(G)2})})\\
&\ge (t-k-1)n+\frac{k^2+3k-3t^2+t}{2}+4~ ~~(\mbox{since $s\le t-1$})\\
&\ge n+\frac{k^2+3k-3(3k-1)^2+k+2}{2}+4~~~(\mbox{since $k+2\le t\le 3k-1$})\\
&= n - 13k^2 + 11k +\frac{7}{2}\\
&> 0~(\mbox{since $ n\ge13k^2-11k-1$}), ~ \mbox{a contradiction.}
\end{aligned}
\end{equation*}
If $t\ge 3k$, then by Claims 2 and 7, we have
\begin{equation*}
\begin{aligned}
0&>\displaystyle\binom{n-k-1}{2}+k+2-e(G^*)\\
&\ge (t-k-1)n+\frac{k^2+5k-t^2-t}{2}+4-(t-k)s-kt ~~~(\mbox{Since inequality (\ref{e(G)2})})\\
&\ge \frac{t^2-t-k^2+k}{2}+3-s-kt  ~~~(\mbox{since $n\ge s+t+k+1$})\\
&\ge\frac{t^2-(2k+3)t-k^2+k}{2}+4~~ ~(\mbox{since $s\le t-1$})\\
&\ge k^2 - 4k + 4~(\mbox{since $t\ge3k$})\\
&\ge 0 ~~~(\mbox{since $k\ge1$}), ~ \mbox{a contradiction.}
\end{aligned}
\end{equation*}
Thus, the claim hold.
$\Box$

By inequality (\ref{dG(v)}), Claim 2 ($s\le t-1=k$), Claim 4 ($q\le1$) and Claim 7 ($t=k+1$), we have
\begin{equation}\label{d(v)2}
\begin{aligned}
\sum_{v\in T}d_{G^*}(v)\le k(t-s)+st+q-2\le k(k+1)+s-1\le k(k+1)+k-1.
\end{aligned}
\end{equation}
Equality holds if and only if equality holds simultaneously in each of the above inequalities.
It is implies that $G^*[S\cup T]=K_{s,k+1}$, $s=k$, $\sum_{v\in T}d_{G^*-S}(v)=k-1$ and $q=1$ (i.e., $G^*-(S\cup T)=C_1\cong K_{n-2k-1}$
and $e_{G^*}(C_1,T)+kc_1\equiv 1 ~(mod ~2)$ where $c_1=n-2k-1$). We also obtain that $G^*[S\cup C_1]= K_k\vee K_{n-2k-1}=K_{n-k-1}.$
Next, we consider separately the adjacency relations between the set $T$ and $C_1$, as well as those within $T$ itself, in $G^*$.

Let $T=\{v_1,\ldots, v_{k+1}\}$, $S=\{u_1,\ldots, u_{k}\}$ and $V(C_1)=\{w_{1},\ldots, w_{n-2k-1}\}$.
Suppose that $x=(x_1,\dots,x_{n})^{T}$ is the Perron vector of $A(G^*)$.  Without loss of generality, assume that $x_{w_{1}}\ge\cdots\ge x_{w_{n-2k-1}}$.
By symmetry and eigenvalue equation $\rho(G^*) x_{u}=A(G^*)x_{u}$ for any $u\in V(G^*)$,
we can easily get that $x_{u_1}=\cdots=x_{u_k}> x_{w_1}\ge\cdots\ge x_{w_{n-2k-1}}$.
Let $d_{C_1}(v_i)=d_i~(\ge0)$, for $v_i\in T$.
Suppose that $v_i\in T$ with $d_i\ge1$, that is $N_{C_1}(v_{i})\neq \emptyset$.
Then we have $N_{C_1}(v_{i})=\{w_{1},\ldots,w_{d_i}\}$.
Otherwise, there exists a vertex $v_j \in T $ ($d_j\ge 1$) such that $v_jw_t\notin E(G^*)$ and  $v_jw_{t'}\in E(G^*)$, where $w_t,w_{t'}\in V(C_1)$ and $1 \le t \le d_{j} < t' \le n-2k-1$.
Note that $x_{w_t}\ge x_{w_{t'}}$, and let $G_1= G^*+v_jw_t-v_jw_{t'}$. Then we can obtain that $\rho(G_1) > \rho(G^*)$
by Lemma \ref{duoling}, which contradicts the maximality of $\rho(G^*)$.
Relabel the vertices in $T$ such that $d_1\ge d_2\ge\cdots\ge d_{k+1}\ge0$.
Recall that $N_{C_1}(v_{i})=\{w_{1},\ldots,w_{d_i}\}$.
Then $N_G(w_{i+1})\backslash\{w_i\}\subseteq N_G(w_i)\backslash\{w_{i+1}\}$, where $i\in [1,d_1]$.
By symmetry and eigenvalue equation $\rho(G^*) x_{w_i}=A(G^*)x_{w_i}$ for any $w_i\in V(C_1)$,
we have $x_{w_{1}}\ge \cdots\ge x_{w_{d_1}}> x_{w_{d_1+1}}=\cdots= x_{w_{n-2k-1}}\triangleq a'$.
Without loss of generality, we can assume that $x_{v^*}=\max\{x_{v_i} |v_i\in T\}$ and let $d_{C_1}(v^*)=d^*$.
Clearly, $d^*\le d_1\le k-1$ and $d_T(v^*)\le \lfloor\frac{k-1}{2}\rfloor$ as $\sum_{v\in T}d_{G^*-S}(v)=k-1$.
Then we have following claim.
{\flushleft\bf Claim 8.} $a'\ge x_{v^*}$.
{\flushleft\bf \emph{Proof}.}
By eigenvalue equation $\rho(G^*)x=A(G^*)x$, we have
\begin{equation*}\label{equ::x_i,b}
\begin{cases}
&\rho(G^*)x_{v^*}=kx_{u_1}+\sum_{i=1}^{d^*}x_{w_i}+\sum_{v_i\in N_T(v^*)}x_{v_i}\\
&\rho(G^*)a'=kx_{u_1}+\sum_{i=1}^{d_1}x_{w_i}+(n-2k-2-d_1)a'.
\end{cases}
\end{equation*}
From above equations, we have
\begin{equation*}
\begin{aligned}
\rho(G^*)(a'-x_{v^*})&=\sum_{i=d^*+1}^{d_1}x_{w_i}\!+\!(n\!-\!2k\!-\!2\!-\!d_1)a'-\!\sum_{v_i\in N_T(v^*)}x_{v_i}\\
&\ge (n\!-\!2k\!-\!2\!-\!d_1)a'\!-\!d_T(v^*)x_{v^*} ~(\mbox{since $x_{v^*}=\max\{x_{v_i} |v_i\in T\}$})\\
&= (n\!-\!2k\!-\!2\!-\!d_1\!-\!d_T(v^*))a'\!+\!d_T(v^*)(a'\!-\!x_{v^*})\\
&\ge\left(n\!-\!\frac{7k\!+\!1}{2}\right)a'\!+\!d_T(v^*)(a'\!-\!x_{v^*}) ~(\mbox{since $d_1\!\le\! k\!-\!1, d_T(v^*)\!\le\! \lfloor\frac{k\!-\!1}{2}\rfloor$}).
\end{aligned}
\end{equation*}
That is
$$(\rho(G^*)-d_T(v^*))(a'-x_{v^*})\ge \left(n-\frac{7k+1}{2}\right)a'\ge0,$$
due to $\rho(G^*)>n-k-2>d_T(v^*)$, $n\ge 13k^2-11k-1>\frac{7k+1}{2}$ when $k\ge2$ and $n\ge 4$ when $k=1$.
It is implies that $a'\ge x_{v^*}$, as required.
$\Box$

For any $v_i\in T$, by Claim 8, we have $x_{w_{1}}\ge x_{w_{d_1}} >x_{w_{d_1+1}}\ge x_{v_i}$. Note that $\sum_{v\in T}d_{G^*-S}(v)=k-1$,
so $2e_{G^*}(T)+e_{G^*}(C_1,T)= k-1$.
If $k=3$, then we have $e_{G^*}(T)=0$, $e_{G^*}(C_1,T)=2$ and thus $G^*=G^1_{n,3}$ due to $G$ is 1-tough graph.
Let $k\ge 4$, then we have $G\in \mathcal{G}_{n,k}$ by Lemmas  \ref{edge}, \ref{duoling} and the proof of Claim 8 (it is implies that $N_{C_1}(v_i)=\{w_1,\ldots,w_{d_i}\}$ for $v_i\in T$).
Therefore, $G^*=G_{n,k}$ by Lemma \ref{main lemma}, which completes the proof.
\end{proof}

\vskip 0.382 in
\noindent
\textbf{\Large Declaration of Competing Interest}

The authors declare that they have no known competing financial interests or personal
relationships that could have appeared to influence the work reported in this paper.

\vskip 0.382 in
\noindent

\section*{Acknowledgements}
This research is supported by China Postdoctoral Science Foundation (No. 2025M783100),
the National Natural Science Foundation of China (Nos. 12271162 and 12401468),
Tian shan Talent Training Program (No. 2024TSYCQNTJ0001),
the Natural Science Foundation of Xinjiang Uygur Autonomous Region, (No. 2023D01C165),
the Natural Science Foundation of Shanghai (No. 22ZR1416300)
and the Program for Professor of Special Appointment (Eastern Scholar) at Shanghai Institutions of Higher Learning (No. TP2022031).

\end{document}